\documentclass[12pt]{amsart}
\usepackage{amssymb}  
\usepackage{latexsym} 
\usepackage[all]{xy}
\usepackage{comment}
\usepackage{url}

\newcommand{\Aff}{{\mathbb A}}
\newcommand{\Aut}{\operatorname{Aut}}
\newcommand{\C}{{\mathbb C}}
\newcommand{\F}{{\mathbb F}}
\newcommand{\Gal}{\operatorname{Gal}}
\newcommand{\GL}{\operatorname{GL}}
\newcommand{\HH}{{\mathcal H}}
\newcommand{\isom}{\cong}
\newcommand{\K}{{\mathbb K}}
\newcommand{\Kbar}{\overline{K}}

\newcommand{\PGL}{\operatorname{PGL}}

\newcommand{\PP}{{\mathbb P}}
\newcommand{\Q}{{\mathbb Q}}
\newcommand{\ra}{\longrightarrow}
\newcommand{\rank}{{\operatorname{rank}}}
\newcommand{\rhobar}{{\overline{\rho}}}
\newcommand{\SL}{\operatorname{SL}}
\newcommand{\x}{{\mathbf x}}
\newcommand{\Z}{{\mathbb Z}}
\newcommand{\ZZ}{{\mathcal Z}}
\newcommand{\zz}{\zeta}

\newenvironment{ProofOf}[1]{\par\noindent{\em Proof of #1.}}%
                       {\hspace*{\fill}\nobreak$\Box$\par\medskip}

\newtheorem{Proposition}{Proposition}[section]
\newtheorem{Theorem}[Proposition]{Theorem}
\newtheorem{Lemma}[Proposition]{Lemma}

\theoremstyle{definition}

\newtheorem{Remark}[Proposition]{Remark}
\newtheorem{Example}[Proposition]{Example}

\begin{document}
\date{29th April 2015}
\title[On families of $9$-congruent elliptic curves]
{On families of $9$-congruent elliptic curves}

\author{Tom~Fisher}
\address{University of Cambridge,
          DPMMS, Centre for Mathematical Sciences,
          Wilberforce Road, Cambridge CB3 0WB, UK}
\email{T.A.Fisher@dpmms.cam.ac.uk}

\renewcommand{\baselinestretch}{1.1}
\renewcommand{\arraystretch}{1.3}

\renewcommand{\theenumi}{\roman{enumi}}

\begin{abstract}
  We compute equations for the families of elliptic curves
  $9$-congruent to a given elliptic curve. We use these to find
  infinitely many non-trivial pairs of $9$-congruent elliptic curves
  over $\Q$, i.e. pairs of non-isogenous elliptic curves over $\Q$
  whose $9$-torsion subgroups are isomorphic as Galois modules.
\end{abstract}

\maketitle

\section*{Introduction}

Elliptic curves $E_1$ and $E_2$ are {\em $n$-congruent} if their
$n$-torsion subgroups $E_1[n]$ and $E_2[n]$ are isomorphic as Galois
modules. They are {\em directly $n$-congruent} if the isomorphism
$\phi : E_1[n] \to E_2[n]$ respects the Weil pairing $e_n$ and {\em
  reverse $n$-congruent} if $e_n(\phi P,\phi Q) = e_n(P,Q)^{-1}$ for
all $P,Q \in E_1[n]$.  The elliptic curves directly and reverse
$n$-congruent to a given elliptic curve $E$ are parametrised by the
modular curves $Y_E(n) = X_E(n) \setminus \{ \text{cusps} \}$ and
$Y_E^-(n) = X_E^-(n) \setminus \{ \text{cusps} \}$.  Finding equations
for these modular curves can help with finding non-trivial pairs of
$n$-congruent elliptic curves.  Some of the potential applications are
described in \cite{BS}, \cite{BHLS}, \cite{CM}, \cite{Frey},
\cite{PSS}.

Equations for $X_E(n)$, and the family of curves it parametrises, have
been computed for $n = 2,3,4,5,6,7,8, 11$ by Rubin and Silverberg
\cite{RubinSilverberg}, \cite{RubSil6}, \cite{RubinSilverberg2},
\cite{Silverberg}, Papadopoulos \cite{Pap6}, Halberstadt and Kraus
\cite{HK}, Chen \cite{Chen8}, and Fisher~\cite{7and11}. The analogous
problem for $n$-congruences that do not respect the Weil pairing has
been solved for the same values of $n$. The additional references for
this include papers by Fisher \cite{g1hess}, \cite{enqI}, Bruin and
Doerksen \cite[Section 7]{BD}, and Poonen, Schaefer and Stoll
\cite[Section 7.2]{PSS}.

In this paper we treat the case $n=9$. That is, we give equations for
$X_E(9)$ and $X^{-}_E(9)$, and for the families of curves they
parametrise. We use these formulae to exhibit non-trivial triples of
$9$-congruent elliptic curves over $\Q$, and non-trivial pairs of
$9$-congruent elliptic curves over $\Q(T)$.  We also give equations
for the modular diagonal quotient surfaces whose points parametrise 
pairs of $9$-congruent elliptic curves.

We work over a field $K$ of characteristic $0$. We write $\Kbar$ for
the algebraic closure, and $\zeta_n$ for a primitive $n$th root of
unity.

\section{Statement of results}
\label{sec:statres}

If $n=3$ or $9$ then every element of $(\Z/n\Z)^\times$ is either plus
or minus a square. Therefore any pair of $n$-congruent elliptic curves are
either directly or reverse $n$-congruent.  In the case $n=3$, we have
$X_E^\pm(3) \isom \PP^1$ and the families of curves parametrised are
as follows.
\begin{Theorem}
\label{MainThm3}
Let $E/K$ be the elliptic curve $y^2 = x^3 - 27 c_4 x - 54 c_6$.  Then
the family of elliptic curves parametrised by $Y^{\pm}_E(3)$ is
\begin{equation}
\label{W3}
 y^2 =  x^3 - 27 A^\pm(r,s)  x - 54 B^\pm(r,s) 
\end{equation}
where
\begin{align*}
  A^+(r,s) & = c_4 r^4 + 4 c_6 r^3 s + 6 c_4^2 r^2 s^2 + 4 c_4 c_6 r
  s^3 - (3 c_4^3 - 4 c_6^2) s^4, \\
  B^+(r,s) & = c_6 r^6 + 6 c_4^2 r^5 s + 15 c_4 c_6 r^4 s^2 + 20 c_6^2
  r^3 s^3 \\ &~\qquad + 15 c_4^2 c_6 r^2 s^4 + 6(3 c_4^4 - 2 c_4
  c_6^2) r s^5 + (9 c_4^3 c_6 - 8 c_6^3) s^6,
\end{align*}
and 
\begin{align*}
  A^-(r,s) & = -4 (r^4 - 6 c_4 r^2 s^2 - 8 c_6 r s^3
  - 3 c_4^2 s^4 ) /(c_4^3-c_6^2), \\
  B^-(r,s) & = -8 B^+(r,s)/(c_4^3-c_6^2)^2.
\end{align*}
\end{Theorem}
\begin{proof} In the direct case this family of curves was first
  computed by Rubin and Silverberg \cite{RubinSilverberg}. The above
  formulae are taken from \cite[Sections 8, 9 and 13]{g1hess}, with
  $(r,s) = (\lambda,\mu)$, respectively $(r,s) = (c_6 \xi + c_4^2
  \eta,-c_4 \xi - c_6 \eta)$.  These formulae are available in Magma
  \cite{Magma} via the function {\tt HessePolynomials}.
\end{proof}

Our new results are in the case $n=9$, where the curves $X^{\pm}_E(9)$
have genus 10. We write these curves as the complete intersection of
two cubics in $\PP^3$.

\begin{Theorem}
\label{MainThm9}
Let $E/K$ be the elliptic curve $y^2 = x^3 + a x + b$.  Then
$X^{\pm}_E(9) = \{F^{\pm}_1 = F^{\pm}_2 = 0 \} \subset \PP^3$ where
\begin{equation*}
\begin{aligned}
 F^+_1(x,y,z,t)  & =  
    x^2 t +  6 x y z + 6 b x t^2 + 6 y^3 - 9 a y^2 t + 6 a^2 y t^2
   - 3 b z^3 \\ & \qquad  + 3 a^2 z^2 t  
         + 9 a b z t^2  - (a^3 - 12 b^2) t^3, \\
 F^+_2(x,y,z,t)  & = 
    x^2 z  +  6 x y^2 - 6 a x y t + 2 a^2 x t^2 - 9 a y^2 z - 18 b y z^2 
 + 12 a^2 y z t \\ & \qquad  + a^2 z^3 + 9 a b z^2 t - 3 a^3 z t^2 + a^2 b t^3,
\end{aligned}
\end{equation*}
and 
\begin{equation*}
\label{XEm9}
\begin{aligned}
 F^{-}_1(x,y,z,t) & =  
    9 x^2 y + 3 x^2 z - 6 a x y t + 6 b x t^2 - 6 a y^3 
      + 27 b y^2 t + 3 a y z^2 \\ & \qquad  + 
        18 b y z t + 3 a^2 y t^2 + a z^3 + 3 b z^2 t + a^2 z t^2 - a b t^3, \\
  F^{-}_2(x,y,z,t) & =    x^3 + 6 a x y z + 18 b x y t + 3 a x z^2 + 6 b x z t 
  + a^2 x t^2 + 9 b y^3 + 6 a^2 y^2 t \\ & \qquad - 9 b y z^2 + 6 a^2 y z t - 3 a b y t^2 
  - 4 b z^3 + 2 a^2 z^2 t + 2 b^2 t^3.
\end{aligned}
\end{equation*}
\end{Theorem}

The families of curves parametrised by $X_E(9)$ and $X_E^-(9)$ are
given by Theorem~\ref{MainThm3} and the following. By abuse of
notation we write $P$ both for a point in $\PP^3$ and for a vector
representing this point.

\begin{Theorem}
\label{thm:geom}
Let $X^{\pm}_E(3)$ and $X^{\pm}_E(9)$ be as described in
Theorems~\ref{MainThm3} and~\ref{MainThm9}, with $a = -27c_4$ and $b =
-54 c_6$.  For $P \in X^{\pm}_E(9)$ with tangent line $P + \lambda Q$
we write $F^{\pm}_i(P+ \lambda Q) = \gamma_i \lambda^2 + \delta_i
\lambda^3$ for $i=1,2$.  Then the forgetful map $X^{\pm}_E(9) \to
X^{\pm}_E(3)$ is given by $(r : s )= (\gamma_2: 3\gamma_1)$.
\end{Theorem}

Let $\ZZ(n)$ be the modular diagonal quotient surface whose points
parametrise pairs of directly $n$-congruent elliptic curves, and
likewise $\ZZ^-(n)$ in the reverse case. It is shown in \cite[Theorem
4]{KS} that the surfaces $\ZZ^\pm(9)$ are each birational over $\C$ to
an elliptic surface. Using Theorem~\ref{MainThm9} we are able to
determine these elliptic surfaces explicitly.

\begin{Theorem}
\label{thm:ellsurf}
The surface $\ZZ(9)$ is birational over $\Q$ to the elliptic surface
\begin{align*}
  y^2 + (6 T^2 + 3 T + 2) x y + T^2 & (T + 1)(4 T^3 + 9 T + 9) y \\ &
  = x^3 - (16 T^4 + 12 T^3 + 9 T^2 + 6 T + 1) x^2.
\end{align*}
The surface $\ZZ^-(9)$ is birational over $\Q$ to the elliptic surface
\begin{align*}
  y^2 + (12 T^3 + 3 T^2 - 6) x y + (T & -1)^3 (T^3-1) (4 T^3 - 3 T -
  7) y \\ & = x^3 - 3 (T+1) (T^3-1) (9 T^2 + 2 T + 1) x^2.
\end{align*}
\end{Theorem}
We use these results to prove \cite[Conjecture 5]{KS} in the case
$n=9$.

\begin{Theorem}
\label{thm:inf}
There are infinitely many non-trivial pairs of directly $9$-congruent
elliptic curves over $\Q$, and these have infinitely many distinct pairs of
$j$-invariants.  By ``non-trivial'' we mean that the elliptic curves
are not isogenous. The same result holds in the reverse case.
\end{Theorem}
\begin{proof}
  For each positive integer $m$ there are finitely many curves on
  $\ZZ^\pm(9)$ corresponding to pairs of elliptic curves related by a
  cyclic isogeny of degree $m$. Each of these curves comes with a
  non-constant morphism to $X_0(m)$ and so has positive genus for $m$
  sufficiently large. The elliptic surfaces in
  Theorem~\ref{thm:ellsurf} each have a $\Q$-rational section of
  infinite order given by $(x,y) = (0,0)$.  Therefore the surfaces
  $\ZZ^\pm(9)$ each contain infinitely many curves birational to
  $\PP^1$ over $\Q$.  Since these curves have genus~$0$, we know by
  the above remarks that only finitely many correspond to pairs of
  isogenous elliptic curves.  It follows that there are non-trivial
  pairs of $9$-congruent elliptic curves over $\Q(T)$.
  In Section~\ref{sec:ex} we give a more explicit version of this first
  part of the proof by using Theorem~\ref{MainThm9} to exhibit some
  non-trivial pairs of $9$-congruent elliptic curves over $\Q(T)$.

  The proof is completed by specialising $T$ to a rational number.  In
  this way we obtain infinitely many pairs of elliptic curves that are
  not $m$-isogenous for any $m \le d$ (for any fixed $d$).  We are
  done by the theorem of Mazur \cite{M}, extended to composite $m$ by
  Kenku \cite{K}, that any cyclic isogeny defined over $\Q$ has degree
  $m \le 163$.
\end{proof}

\section{The modular curves $X(3)$ and $X(9)$}
\label{sec:X3andX9}

Let $M = \mu_n \times \Z/n\Z$ equipped with the pairing
\begin{equation*}
  \langle (\zeta,a) , (\xi,b) \rangle = \zeta^{b} \xi^{-a}. 
\end{equation*}
The modular curve $Y(n) = X(n) \setminus \{ \text{cusps} \}$
parametrises pairs $(E,\phi)$ where $E$ is an elliptic curve and $\phi
: E[n] \to M$ is a symplectic isomorphism, i.e. an isomorphism that
matches up the Weil pairing on $E[n]$ with the pairing
$\langle~,~\rangle$ on $M$.  We identify $\SL_2(\Z/n\Z)$ with the
group of symplectic automorphisms of $M$. There is then a natural
action of $\SL_2(\Z/n\Z)/ \{\pm I\}$ on $X(n)$ with quotient the
$j$-line.

In the case $n=3$ it is well known that $Y(3) = \Aff^1 \setminus \{
t^3 = 1\}$ parametrises the non-singular fibres in the Hesse pencil of
plane cubics:
\begin{equation}
\label{HessePencil}
\{ x^3+y^3+z^3 - 3 t x y z  = 0 \} \subset \PP^2.
\end{equation}
We need an analogous result in the case $n=9$.

\begin{Lemma}
\label{prop:X9}
The modular curve $X(9)$ has equations
\begin{equation}
\label{eqn:X9}
X(9) = \left\{ 
\begin{aligned}
a^2 b +b^2 c +c^2 a & = 0 \\
a b^2 + b c^2 + c a^2 & = d^3 
\end{aligned} \right\} \subset \PP^3. 
\end{equation}
Moreover the forgetful map $X(9) \to X(3)$ is given by
\begin{equation}
\label{eqn:forget}
t = -(a^3 + b^3 + c^3 + 6 abc)/(3 d^3).
\end{equation}
\end{Lemma}

\begin{proof}
  In \cite[Section 2]{7and11} we showed, following work of Klein,
  V\'elu and others, that if $n \ge 5$ is an odd integer then $Y(n)
  \subset Z(n)$ where $Z(n) \subset \PP^{n-1}$ is defined as
  follows. We take co-ordinates $(a_0:a_1: \ldots: a_{n-1})$ on
  $\PP^{n-1}$ and agree to read all subscripts mod $n$. Then $Z(n)$ is
  defined by $a_0=0$, $a_{n-i}=-a_i$ and
\begin{equation*}
 \rank ( a_{i-j} a_{i+j} )_{i,j=0}^{n-1} \le 2.
\end{equation*}

In the case $n=9$ we put
\begin{equation}
\label{zero}
(a_0:a_1: \ldots: a_8) = (0: a: -b: d : c : -c : -d : b : -a). 
\end{equation}
 Then $Z(9) \subset \PP^3$ is defined by 
\[ \rank \begin{pmatrix}
  0   & -a^2 & -b^2 & -d^2 &-c^2 \\
  a^2 &    0 &  -ad &   bc &  cd \\
  b^2 &   ad &    0 &   ac & -bd \\
  d^2 &  -bc &  -ac &    0 & -ab \\
  c^2 &  -cd &   bd &   ab &   0 \\
\end{pmatrix} \le 2, \]
equivalently
\begin{align*}
(a^2 b+b^2 c+c^2 a) d & = 0 & b c^3 - b a^3 - c d^3 & = 0 \\
a b^3 - a c^3 - b d^3 & = 0 & c a^3 - c b^3 - a d^3 & = 0.
\end{align*}
Adding together the last three equations and factoring shows that
$Z(9)$ is the union of the curve defined in~(\ref{eqn:X9}) and four
isolated points
\[ (0:0:0:1), \, \, (1:1:1:0), \, \, (1:\zeta_3:\zeta_3^2:0), \, \,
(1:\zeta_3^2:\zeta_3:0). \] Since $X(9)$ is a curve, it must therefore
be as defined in~(\ref{eqn:X9}).

Now writing $(x_0: x_1 : \ldots : x_{n-1})$ for our co-ordinates on
$\PP^{n-1}$, and again agreeing to read all subscripts mod $n$, it is
shown in \cite[Section 3]{pfaff} that the family of elliptic curves
parametrised by $Y(n)$ is defined by
\[ \rank ( a_{i-j} x_{i+j} )_{i,j=0}^{n-1} \le 2. \] 
In the case $n=9$ the elliptic curve $E$ corresponding to $(a:b:c:d)
\in Y(9)$ is the curve of degree $9$ in $\PP^8$ defined by the
equations
\begin{align*}
    a d x_0^2 - b^2 x_1 x_8 + a^2 x_2 x_7 &= 0 \\
    b c x_0^2 + d^2 x_1 x_8 - a^2 x_3 x_6 &= 0 \\
    c d x_0^2 + c^2 x_1 x_8 - a^2 x_4 x_5 &= 0
\end{align*}
and their cyclic permutes. Thus $E$ is defined by a total of 27
quadrics.

Let $0_E$ be the point on $E$ given by~(\ref{zero}).  In principle we
could complete the proof by putting the elliptic curve $(E,0_E)$ in
Weierstrass form.  However it is simpler to argue as follows. The
action of $E[9]$ on $E$ by translation is generated by the maps $x_i
\mapsto x_{i+1}$ and $x_i \mapsto \zeta_9^i x_i$.  From this we see
that the morphism
\[ (x_0: x_1: \ldots :x_8) \mapsto (x:y:z) = (x_0 x_3 x_6 : x_1 x_4
x_7 : x_2 x_5 x_8) \] 
quotients out by the action of $E[3]$. We may
therefore identify this morphism with the multiplication-by-$3$ map on
$E$.  On the other hand, we find by direct calculation that the image
takes the form~(\ref{HessePencil}) with $t$ as specified
in~(\ref{eqn:forget}). The forgetful map $X(9) \to X(3)$ is therefore
as claimed.
\end{proof}

\begin{Remark} 
\label{rem:act}
The action of $\SL_2(\Z/9\Z)$ on $X(9) \subset \PP^3$ is described by
a projective representation $\rhobar : \SL_2(\Z/9\Z) \to
\PGL_4(\Kbar)$.  According to \cite[Section 2]{7and11} it is given on
the generators $S = (\begin{smallmatrix} 0 & 1 \\ -1 &
  0 \end{smallmatrix})$ and $T = (\begin{smallmatrix} 1 & 1 \\ 0 &
  1 \end{smallmatrix})$ for $\SL_2(\Z/9\Z)$ by
\[ \rhobar(S) = \begin{pmatrix}
  \zz_9     - \zz_9^{8} & \zz_9^{7} - \zz_9^{2} &  \zz_9^{4} - \zz_9^{5} &  \zz_9^{3} - \zz_9^{6} \\
  \zz_9^{7} - \zz_9^{2} & \zz_9^{4} - \zz_9^{5} &  \zz_9 - \zz_9^{8} &  \zz_9^{3} - \zz_9^{6} \\
  \zz_9^{4} - \zz_9^{5} & \zz_9 - \zz_9^{8} & \zz_9^{7} - \zz_9^{2} &  \zz_9^{3} - \zz_9^{6} \\
  \zz_9^{3} - \zz_9^{6} & \zz_9^{3} - \zz_9^6 & \zz_9^{3} - \zz_9^{6}
  & 0
\end{pmatrix} 
 \text{ and } \,\,
\rhobar(T) = \begin{pmatrix}
\zz_9 & 0 & 0 & 0 \\
0 & \zz_9^4 & 0 & 0 \\
0 & 0 & \zz_9^7 & 0 \\
0 & 0 & 0 & \zz_9^6 
\end{pmatrix}. \]
In particular, the action of 
$\ker \big(\SL_2(\Z/9\Z) \to \SL_2(\Z/3\Z)\big) \isom (\Z/3\Z)^3$ 
is generated by
\begin{align*}
(a:b:c:d) &\mapsto ( a : b : c : \zz_3 d ) \\
(a:b:c:d) &\mapsto ( b : c : a :  d ) \\
(a:b:c:d) &\mapsto ( \zz_3 a + b + c : a + \zz_3 b + c :
    a + b + \zz_3 c : (\zz_3 - 1) d ).
\end{align*}
It may be checked that the map $X(9) \to X(3)$ in Lemma~\ref{prop:X9}
quotients out by this action.
\end{Remark}

The pencil of cubics defining $X(9) \subset \PP^3$ is naturally a copy
of $X(3) \isom \PP^1$. To explain this, let $F_1 = a^2b + b^2 c + c^2
a$ and $F_2 = a b^2 + b c^2 + c a^2 - d^3$ be the cubics defining
$X(9)$. We write $\HH(F)$ for the Hessian matrix of a form $F$, that
is, the matrix of second partial derivatives. Then
\begin{equation}
\label{hess-id}
 \det \HH(t F_1 - F_2) 
    = -48 (t^3 - 1) (a^3 + b^3 + c^3 - 3 a b c) d. 
\end{equation}
Therefore the Hessian vanishes for just $4$ cubics in the pencil, and
these correspond to the cusps of $X(3)$. The forgetful map $X(9) \to
X(3)$ has the following geometric description.

\begin{Lemma}
\label{lem:geom}
For $P \in X(9)$ with tangent line $P + \lambda Q$ we write $F_i(P+
\lambda Q) = \gamma_i \lambda^2 + \delta_i \lambda^3$ for
$i=1,2$. Then the forgetful map $X(9) \to X(3)$ is given by $t =
\gamma_2/\gamma_1$.
\end{Lemma}

\begin{proof}
  We temporarily write $a_1, a_2, a_3, a_4$ for $a,b,c,d$ and let
  $\Lambda$ be the $4 \times 4$ alternating matrix with entries
  \[ \Lambda_{ij} = \frac{\partial F_1}{\partial a_k} \frac{\partial
    F_2}{\partial a_l} -\frac{\partial F_1}{\partial a_l}
  \frac{\partial F_2}{\partial a_k} \] where $(i,j,k,l)$ is an even
  permutation of $(1,2,3,4)$.  Then specialising $\Lambda$ at $P =
  (a:b:c:d) \in X(9)$ gives a matrix whose rows (or columns) span the
  tangent line at $P$.  Let $D = (a^3 + b^3 + c^3 - 3abc)d$.  We find
  by direct calculation that
\[ \Lambda \,\, \HH(F_i)
\,\, \Lambda \equiv \gamma_i  \, D 
\begin{pmatrix} a^2 & ab & ac & ad \\
ab & b^2 & bc & bd \\ ac & bc & c^2 & cd \\ ad & bd & cd & d^2 
\end{pmatrix}
\mod{(F_1,F_2)} 
\]
for $i=1,2$, where $\gamma_1 = -18 d^3$ and $\gamma_2 = 6 (a^3 + b^3 +
c^3 + 6 a b c)$.  By Lemma~\ref{prop:X9} we have $t =
\gamma_2/\gamma_1$.
\end{proof}

\section{Remarks on twisting}

In this section we make some general remarks on computing twists of
$X(n)$. These will be used to find equations for $X_E^-(9)$ once we
have found equations for $X_E(9)$ by another method.  We suppose that
$X(n)$ has been embedded in (and spans) $\PP^{N-1}$ and that the
action of $\SL_2(\Z/n\Z)$ is given by
\begin{equation*}
  \rhobar : \SL_2(\Z/n\Z)  \to \PGL_N(\Kbar). 
\end{equation*}
We write $\propto$ for equality in $\PGL_N(\Kbar)$, and a superscript
$-T$ for the inverse transpose of a matrix.  Let $\iota =
(\begin{smallmatrix} 1 & 0 \\ 0 & -1 \end{smallmatrix})$.  We further
suppose that
\begin{equation*}
  \rhobar (\iota \gamma \iota) \propto 
  \rhobar(\gamma)^{-T} 
\end{equation*}
for all $\gamma \in \SL_2(\Z/n\Z)$. This condition is satisfied in the
case of $X(9) \subset \PP^3$ since the matrices $\rhobar(S)$ and
$\rhobar(T)$ in Remark~\ref{rem:act} are symmetric.
 
\begin{Lemma}
\label{TwistLem}
Let $E/K$ be an elliptic curve and $\phi : E[n] \to M$ a symplectic
isomorphism defined over $\Kbar$.  Suppose $h \in \GL_N(\Kbar)$
satisfies
\[\sigma(h) h^{-1} \propto \rhobar ( \sigma (\phi) \phi^{-1})
\] for all $\sigma \in \Gal(\Kbar/K)$. Then $X_E(n) \subset \PP^{N-1}$
and $X^-_E(n) \subset \PP^{N-1}$ are the twists of $X(n) \subset
\PP^{N-1}$ given by $ X_E(n) \isom X(n) ;\, \x \mapsto h \x$ and $
X^-_E(n) \isom X(n) ;\, \x \mapsto h^{-T} \x.$
Moreover these isomorphisms are compatible with the maps to the $j$-line.
\end{Lemma}
\begin{proof}
This is a special case of \cite[Lemma 3.2]{7and11}. 
\end{proof}

The following lemma generalises \cite[Proposition 7.5]{PSS}.

\begin{Lemma}
\label{lem:dual}
Let $E/K$ be an elliptic curve. If $h \in \GL_N(\Kbar)$ describes
$X_E(n) \subset \PP^{N-1}$ as a twist of $X(n) \subset \PP^{N-1}$ via
$ X_E(n) \isom X(n) ;\, \x \mapsto h \x$, and this isomorphism is
compatible with the maps to the $j$-line, then $X^-_E(n) \subset
\PP^{N-1}$ is the twist of $X(n) \subset \PP^{N-1}$ via $X^-_E(n)
\isom X(n) ;\, \x \mapsto h^{-T} \x.$
\end{Lemma}
\begin{proof}
  If $h$ is the same as in Lemma~\ref{TwistLem} then the result is
  clear.  In general if $h_1$ and $h_2$ are two such maps then there
  is a commutative diagram
  \[ \xymatrix{ X_E(n) \ar@{=}[d] \ar[r]^{h_1} & X(n) \ar[d]^{\beta} \\
  X_E(n) \ar[r]^{h_2} & X(n) } \] 
  where $\beta$ is an automorphism of $X(n)$ defined over $\Kbar$.
  Since $h_1$ and $h_2$ are compatible with the maps to the $j$-line,
  it follows that $h_2 \propto \rhobar(\gamma) h_1 \alpha^{-1}$ for
  some $\gamma \in \SL_2(\Z/n\Z)$ and $\alpha \in \GL_N(K)$. The
  change from $h_1$ to $h_2$ makes no difference to the desired
  conclusions.
\end{proof}

\section{Formulae in the case of a rational $3$-torsion point}

In this section we find formulae for $X^{\pm}_E(9)$ and $X^{\pm}_E(9)
\to X^{\pm}_E(3)$ in the case $E$ has a rational $3$-torsion point.

\begin{Lemma} 
\label{lem:3tors}
Let $E/K$ be an elliptic curve with a rational $3$-torsion point $T$,
and discriminant $\Delta$. Then $E$ has a Weierstrass equation of the
form
\begin{equation}
\label{Weqn3}
y^2 + a_1 xy + a_3 y = x^3
\end{equation} with $T= (0,0)$.
Moreover $K(E[3]) = K(\zeta_3, \sqrt[3]{\Delta})$.
\end{Lemma}
\begin{proof}
  Moving $T$ to $(0,0)$ gives a Weierstrass equation $y^2 + a_1 xy +
  a_3 y = x^3 + a_2 x^2 + a_4 x$.  Since $2T \not= 0$ we have $a_3
  \not= 0$. By a substitution $y \leftarrow y + r x$ we may assume
  $a_4 = 0$. Then $E$ meets the line $y=0$ in divisor
  $2(0,0)+(-a_2,0)$. So for $3T = 0$ we need $a_2 =0$. For the last
  statement we note that $\Delta = a_3^3 (a_1^3 - 27a_3)$ and $E[3]$
  has basis $T = (0,0)$ and $T' = ( 3 a_3/(\delta - a_1), a_3 (\zeta_3
  \delta - a_1 )/(\delta - a_1))$ where $\delta = \sqrt[3]{a_1^3 -
    27a_3}$.
\end{proof}

\begin{Lemma} 
\label{lem:E3}
Let $E/K$ be an elliptic curve with a rational $3$-torsion point, and
discriminant $\Delta$.
\begin{enumerate} \item The family of elliptic curves parametrised by
  $Y(3) = \Aff^1 \setminus \{ t^3 = 1 \}$ is \[y^2 + 3 t x y + (t^3 -
  1) y = x^3.\] \item The family of elliptic curves parametrised by
  $Y_E^{\pm}(3) = \Aff^1 \setminus \{ t^3 = \Delta^{\pm 1} \}$ is
\[y^2 + 3 t x y + (t^3 - \Delta^{\pm 1}) y = x^3.\]
\end{enumerate}
\end{Lemma}
\begin{proof}
  The family of elliptic curves in (i) is the same as that
  in~(\ref{HessePencil}), after a quadratic twist by $-3$ or a
  substitution $t \leftarrow (t+2)/(t-1)$.  Let $c_4, c_6, \Delta$ be
  the usual quantities associated to the Weierstrass
  equation~(\ref{Weqn3}). Then part (ii) is the special case of
  Theorem~\ref{MainThm3} with $(r,s) = (a_1^2 t - a_1^3 a_3 + 36
  a_3^2, t - a_1a_3),$ respectively $(r,s) = ((a_1^3 a_3 - 36 a_3^2) t
  + 2 a_1^2, a_1 a_3 t + 2)$.  Alternatively the lemma may be proved
  directly by an argument similar to the proof of
  Lemma~\ref{lem:3tors}.
\end{proof}

For $a \in K$ we write $\sqrt[3]{a}$ for the image of $x$ in
$K[x]/(x^3 - a)$. Each equation involving $\sqrt[3]{a}$ in the next
theorem is written as a short-hand for three equations, one for each
$K$-algebra homomorphism $K[x]/(x^3 - a) \to \Kbar$.

\begin{Theorem}
\label{prop:main3}
Let $E/K$ be the elliptic curve
\begin{equation}
\label{weqn3}
y^2 + 3 t_E x y + (t_E^3 - \Delta) y = x^3.
\end{equation}
Then $X(9)$ and $X_E^{\pm}(9)$ have equations in $\Aff^4$ (with
co-ordinates $t,u,v,w$) given by
\begin{align}
\label{eqn:X(9)}
&\hspace{2em}\,\, X(9) && t - \sqrt[3]{1}  = 9 \left( u + v \sqrt[3]{1} + w  (\sqrt[3]{1})^2 \right)^3, \\
\label{eqn:XE9}
&\hspace{2em}X_E(9) && t - \sqrt[3]{\Delta}  = (t_E - \sqrt[3]{\Delta}) \left( u + v \sqrt[3]{\Delta} + w  (\sqrt[3]{\Delta})^2 \right)^3, \\
\label{eqn:XEm9}
&\hspace{2em}X_E^-(9) && t - (\sqrt[3]{\Delta})^{-1}  = \frac{3}{t_E - \sqrt[3]{\Delta}} \left( u + v \sqrt[3]{\Delta} + w  (\sqrt[3]{\Delta})^2 \right)^3.
\end{align}
Moreover the maps to $X(3)$ and $X_E^{\pm}(3)$ are given by 
$(t,u,v,w) \mapsto t$, where $t$ is the co-ordinate in Lemma~\ref{lem:E3}.
\end{Theorem}

\begin{Remark}
\label{rem:expand}
The equations in Theorem~\ref{prop:main3} may be re-written as
follows. First we expand and equate coefficients of $1, \sqrt[3]{1},
(\sqrt[3]{1})^2$, respectively $1, \sqrt[3]{\Delta},
(\sqrt[3]{\Delta})^2$.  We then eliminate $t$, and homogenise
(introducing a new variable $s$) to give
\begin{equation*}
 X(9) = \left\{ \begin{aligned}  u^2 v + v^2 w + w^2 u &= s^3 \\
  u^2 w + v^2 u + w^2 v &= 0 \end{aligned} \right\} \subset \PP^3 
\end{equation*}
\begin{equation*}
 X_E(9) = \left\{ \begin{aligned}  f_0(u,v,w) - t_E f_1(u,v,w) &= s^3 \\
 f_1(u,v,w) - t_E f_2(u,v,w) &= 0   \end{aligned} \right\} \subset \PP^3 
\end{equation*}
\begin{equation*}
 X_E^-(9) =  \left\{ \begin{aligned}  f_0(u,v,w) + t_E f_1(u,v,w) &= 9 s^3 \\
\Delta f_2(u,v,w) + 9 t_E s^3 &= 0 
\end{aligned} \right\} \subset \PP^3  
\end{equation*}
where
\begin{align*}
f_0(u,v,w) &= u^3 + \Delta v^3 + \Delta^2 w^3 + 6 \Delta u v w \\
f_1(u,v,w) &= 3 (u^2 v + \Delta v^2 w + \Delta w^2 u ) \\
f_2(u,v,w) &= 3 (u^2 w +  v^2 u  + \Delta w^2 v ).
\end{align*}

\end{Remark}

\begin{ProofOf}{Theorem~\ref{prop:main3}}
  The formulae for $X(9)$ and $X(9) \to X(3)$ were already established
  in Lemma~\ref{prop:X9}.

 Let $\K_n$ be the function field of $X(n)$ over $\Kbar$. Let $B
 \subset \K_3^\times/(\K_3^\times)^3$ be the subgroup generated by all
 rational functions on $X(3)$ with zeros and poles only at the cusps.
 Since $X(3) \isom \PP^1$ has $4$ cusps, $B$ has dimension $3$ as an
 $\F_3$-vector space. By~(\ref{eqn:X(9)}) we have $\K_9 = \K_3(
 \sqrt[3]{B})$.  So if $t$ is the co-ordinate on $X_E^{\pm}(3)$ in
 Lemma~\ref{lem:E3} then $X_E^{\pm}(9)$ has equations
 \[
t - (\sqrt[3]{\Delta})^{\pm 1}  = c_{\pm}(E) 
   \left( u + v \sqrt[3]{\Delta} + w  (\sqrt[3]{\Delta})^2 \right)^3 
   \]
   for some constant $c_{\pm}(E)$. Since on $X_E(9)$ there is a
   rational point (corresponding to $E$) above the point $t = t_E$ on
   $X_E(3)$, we can take $c_+(E) = t_E - \sqrt[3]{\Delta}$.

   The equations~(\ref{eqn:X(9)}) and~(\ref{eqn:XE9}) for $X(9)$ and
   $X_E(9)$ differ by a change of co-ordinates defined over
   $\Kbar$. In writing down this change of co-ordinates it is
   important to remember that each of~(\ref{eqn:X(9)})
   and~(\ref{eqn:XE9}) is really three equations.  Let
   $\alpha,\beta,\gamma,\delta \in \Kbar$ satisfy
   \[ \alpha^3 = \frac{t_E - \delta}{9 \delta}, \quad \beta^3 =
   \frac{t_E - \zeta_3 \delta}{9 \zeta_3 \delta}, \quad \gamma^3 =
   \frac{t_E - \zeta^2_3 \delta}{9 \zeta^2_3 \delta}, \quad \delta^3 =
   \Delta. \] Then an isomorphism $X_E(9) \to X(9)$ is given by
   \[ \begin{pmatrix} u \\ v \\ w \end{pmatrix} \mapsto
   \begin{pmatrix} 1 & 1 & 1 \\ 1 & \zeta_3 & \zeta_3^2 \\
     1 & \zeta_3^2 & \zeta_3 \end{pmatrix}^{-1}
   \begin{pmatrix} \alpha & & \\ & \beta & \\ & & \gamma \end{pmatrix}
   \begin{pmatrix} 1 & 1 & 1 \\ 1 & \zeta_3 & \zeta_3^2 \\
     1 & \zeta_3^2 & \zeta_3 \end{pmatrix}
   \begin{pmatrix} 1 & & \\ & \delta & \\ & & \delta^2 \end{pmatrix}
   \begin{pmatrix} u \\ v \\ w \end{pmatrix}.\] 

   By Lemma~\ref{lem:dual} an isomorphism $X^-_E(9) \to X(9)$ is given
   by \[ \begin{pmatrix} u \\ v \\ w \end{pmatrix} \mapsto
   \begin{pmatrix} 1 & 1 & 1 \\ 1 & \zeta_3^2 & \zeta_3 \\
     1 & \zeta_3 & \zeta_3^2 \end{pmatrix}^{-1}
   \begin{pmatrix} \alpha^{-1} & & \\ & \beta^{-1} & \\ & &
     \gamma^{-1}
   \end{pmatrix}
   \begin{pmatrix} 1 & 1 & 1 \\ 1 & \zeta_3^2 & \zeta_3 \\
     1 & \zeta_3 & \zeta_3^2 \end{pmatrix}
   \begin{pmatrix} 1 & & \\ & \delta^{-1} & \\ & &
     \delta^{-2} \end{pmatrix}
   \begin{pmatrix} u \\ v \\ w \end{pmatrix}.\] Therefore $X_E^-(9)$
   has equations
   \begin{equation}
     \label{eqn:XEm9alt}
     t - (\sqrt[3]{\Delta})^{-1}  = \frac{9^2}{t_E - \sqrt[3]{\Delta}}
     \left( u + v (\sqrt[3]{\Delta})^{-1} + w  
          (\sqrt[3]{\Delta})^{-2} \right)^3. 
   \end{equation}
   The equations~(\ref{eqn:XEm9}) and~(\ref{eqn:XEm9alt}), though
   different, are related by a change of co-ordinates defined over
   $K$.
\end{ProofOf}

\section{Completion of proofs}

In this section we complete the proofs of Theorems~\ref{MainThm9},
\ref{thm:geom} and \ref{thm:ellsurf}.  We start by using the results
of the last section to prove Theorems~\ref{MainThm9}
and~\ref{thm:geom} in the case $E(K)[3] \not= 0$. By
Lemma~\ref{lem:3tors} we may assume $E$ takes the
form~(\ref{weqn3}). Then $E$ has shorter Weierstrass equation $y^2 =
x^3 + a x + b$ where
\begin{align*}
 a &= -24 \Delta t_E - 3 t_E^4, \\
 b &= 16 \Delta^2 + 40 \Delta t_E^3 - 2 t_E^6.
\end{align*}
The equations for $X_E(9)$ in Theorem~\ref{MainThm9} and
Remark~\ref{rem:expand} are related by
\[\begin{pmatrix} u \\ v \\ w \\ s \end{pmatrix}
=\begin{pmatrix}
1 &  -3 t_E^2 &  12 \Delta t_E - 3 t_E^4 &  36 \Delta t_E^3 - 9 t_E^6 \\
0 &  -12 t_E &  24 \Delta + 12 t_E^3 &  -216 \Delta t_E^2 \\
0 &  -12 &  36 t_E^2 &  -144 \Delta t_E - 72 t_E^4 \\
1 &  9 t_E^2 &  -36 \Delta t_E + 9 t_E^4 &  96 \Delta^2 + 132 \Delta t_E^3 + 15 t_E^6 \end{pmatrix}
\begin{pmatrix} x \\ y \\ z \\ t \end{pmatrix}. \] The equations for
$X^{-}_E(9)$ in Theorem~\ref{MainThm9} and Remark~\ref{rem:expand} are
related by
\[ \hspace{-1em} \begin{pmatrix} u \\ v \\ w \\ s \end{pmatrix}
=\begin{pmatrix}
  \Delta t_E &  -12 \Delta^2 + 12 \Delta t_E^3 &  -4 \Delta^2 + 7 \Delta t_E^3  & -12 \Delta^2 t_E^2 + 3 \Delta t_E^5 \\
  -2 \Delta &  0  & -6 \Delta t_E^2 &  18 \Delta t_E^4 \\
  t_E^2  & 0 &  4 \Delta t_E - t_E^4 &  -16 \Delta^2 + 8 \Delta t_E^3 - t_E^6 \\
  -\Delta t_E & -4 \Delta^2 + 4 \Delta t_E^3 & -4 \Delta^2 + \Delta
  t_E^3 & 12 \Delta^2 t_E^2 - 3 \Delta t_E^5 \end{pmatrix}
\begin{pmatrix} x \\ y \\ z \\ t \end{pmatrix}. \hspace{-1em} \] These
$4$ by $4$ matrices have determinants $-2^{10} 3^3 (t_E^3-\Delta)^3$
and $2^{10} \Delta^3 (t_E^3-\Delta)^4$, and so are non-singular by the
discriminant condition for $E$.  Theorem~\ref{MainThm9} in the case
$E(K)[3] \not= 0$ now follows from Theorem~\ref{prop:main3}.

Theorem~\ref{thm:geom} follows almost immediately from
Lemma~\ref{lem:geom}. The one detail we must check is how the pencil
of cubics defining $X_E^{\pm}(9)$ in Theorem~\ref{MainThm9} matches up
with the co-ordinates $(r:s)$ on $X_E^{\pm}(3) \isom \PP^1$ in
Theorem~\ref{MainThm3}.  Computing the discriminant of the Weierstrass
equation~(\ref{W3}), we find that the cusps of $X^{\pm}_E(3)$ are the
roots of $f_{\pm}(r,s)=0$ where
\begin{align*}
f_+(r,s) &= 
    r^4 - 6 c_4 r^2 s^2 - 8 c_6 r s^3 - 3 c_4^2 s^4 \\
f_-(r,s) &= c_4 r^4 + 4 c_6 r^3 s + 6 c_4^2 r^2 s^2 
+ 4 c_4 c_6 r s^3 - (3 c_4^3 - 4 c_6^2) s^4.
\end{align*}
On the other hand, writing $\HH(F)$ for the Hessian matrix of $F$, we find that
\[ \det \HH( 3 r F^{\pm}_1 - s F^{\pm}_2 ) = f_{\pm}(r,s) D_{\pm}
(x,y,z,t) \] for some quartic form $D_{\pm} (x,y,z,t)$.  Comparing
with~(\ref{hess-id}) and Lemma~\ref{lem:geom} it follows that the
forgetful map is as claimed in Theorem~\ref{thm:geom}.

To extend the proofs of Theorems~\ref{MainThm9} and~\ref{thm:geom} to
the case $E(K)[3]=0$ we use the following two lemmas.

\begin{Lemma} 
\label{lem:cores}
Let $X,Y,Y'$ be curves defined over $K$. Suppose there
is a commutative diagram
\[ \xymatrix{ Y \ar[d]_{\pi} \ar[r]^{\psi} & Y' \ar[d]^{\pi'} \\ X
  \ar@{=}[r] & X } \] where $\pi$ and $\pi'$ are morphisms defined
over $K$, and $\psi$ is an isomorphism defined over a finite extension
$L/K$.  Suppose that $\pi$ is the map that quotients out by a finite
$K$-rational subgroup $A \subset \Aut(Y)$. If $[L:K]$ and $|A|$ are
coprime then $Y$ and $Y'$ are isomorphic over $K$.
\end{Lemma}
\begin{proof}
  The curve $Y'$ is the twist of $Y$ by a class $\xi \in H^1(K,A)$
  whose restriction to $H^1(L,A)$ is trivial. Since the composition
\[  H^1(K,A) \stackrel{\operatorname{res}}{\ra} H^1(L,A) 
\stackrel{\operatorname{cores}}{\ra} H^1(K,A) \]
is multiplication by $[L:K]$, and this degree is coprime to $|A|$,  
it follows that $\xi$ is trivial.
\end{proof}

\begin{Lemma}
\label{lem:coprime}
Let $E/K$ be an elliptic curve and $p$ a prime. Then $E(L)[p] \not =
0$ for some finite extension $L/K$ with $[L:K]$ coprime to $p$.
\end{Lemma}
\begin{proof} The orbits sizes for the action of Galois on $E[p]
  \setminus \{0\}$ add up to $p^2 - 1$. Therefore some orbit has size
  coprime to $p$.
\end{proof}

To prove Theorem~\ref{MainThm9} we use Lemma~\ref{lem:coprime} to find
$L/K$ an extension of degree coprime to $3$ with $E(L)[3] \not= 0$. We
already know that the theorem is true over $L$. We apply
Lemma~\ref{lem:cores} with $X = X_E^\pm(3)$, $Y = X_E^{\pm}(9)$ and
$Y' \subset \PP^3$ the curve defined in the statement of
Theorem~\ref{MainThm9}. We further take $\pi$ to be the forgetful map,
and $\pi'$ the map defined in the statement of Theorem~\ref{thm:geom}.
It follows that Theorem~\ref{MainThm9} is true over $K$. The proof of
Theorem~\ref{thm:geom} goes through exactly the same as before.

\begin{Remark}
\label{rem:scale}
If elliptic curves $E$ and $E'$ are quadratic twists, then it is easy
to see that the modular curves $X^\pm_E(n)$ and $X^\pm_{E'}(n)$ are
isomorphic.  With $F_1^\pm$ and $F_2^\pm$ as defined in
Theorem~\ref{MainThm9}, this is borne out by the identities
\begin{align*}
  F^+_1(\lambda^2 a,\lambda^3 b; 
  \lambda^3 x,\lambda^2 y,\lambda z,t ) & = \lambda^6 F^+_1(a,b; x,y,z,t ), \\
  F^+_2(\lambda^2 a,\lambda^3 b; 
  \lambda^3 x,\lambda^2 y,\lambda z,t ) & = \lambda^7 F^+_2(a,b; x,y,z,t ),
\end{align*}
and 
\begin{align*}
  F^-_1(\lambda^2 a,\lambda^3 b; 
  \lambda^2 x,\lambda y,\lambda z,t ) & = \lambda^5 F^-_1(a,b; x,y,z,t ), \\
  F^-_2(\lambda^2 a,\lambda^3 b; 
  \lambda^2 x,\lambda y,\lambda z,t ) & = \lambda^6 F^-_2(a,b; x,y,z,t ).
\end{align*}
\end{Remark}

\medskip

\begin{ProofOf}{Theorem~\ref{thm:ellsurf}}
  Treating $a$ and $b$ as additional variables, the equations for
  $X^\pm_E(9)$ in Theorem~\ref{MainThm9} define a threefold.  Our task
  is to quotient out by the action of ${\mathbb G}_m$ in
  Remark~\ref{rem:scale}, to give a surface birational to
  $\ZZ^{\pm}(9)$.  One way to do this is by setting $a=b$, but this
  gives a highly singular model for $\ZZ^{\pm}(9)$, and does not seem
  to help in finding an elliptic fibration. Instead we make the
  following substitutions.

  We start with the direct case.  Putting $a = 12 s - 3 w^2$ and $b =
  u s + 2 w^3$ in Theorem~\ref{MainThm9}, and expanding in powers of
  $s$, we find
  \begin{align*}
    F^+_1(6 v s - 3 w^3,2 s T - w^2,w,1) &= 12 s^2 (g_0 + 4 s g_1), \\
    F^+_2(6 v s - 3 w^3,2 s T - w^2,w,1) &= 36 s^2 (h_0 + 4 s h_1),
  \end{align*}
  where
  \begin{align*}
    g_0 &= u^2 + 3 u v + 3 v^2 + 9 u w + 6 T v w + 3 (T^2 - 12 T + 24) w^2, \\
    g_1 &= T^3 - 9 T^2 + 36 T - 36, \\
    h_0 &= v^2w - (T - 1) u w^2 + 2 (T - 6) v w^2 + (T^2 - 24 T + 48) w^3, \\
    h_1 &= u + (T^2 - 6 T + 12) v - 3 (T - 2) (T - 6) w.
  \end{align*}
  The plane cubic $\{ g_0 h_1 - g_1 h_0 = 0 \} \subset \PP^2$ is a
  smooth curve of genus one, defined over $\Q(T)$, with rational point
  $(u:v:w) = (12:T-6:-1)$. The elliptic surface in
  Theorem~\ref{thm:ellsurf} is obtained by putting this curve in
  Weierstrass form, for example using the method in
  \cite[Chapter~8]{Ca}.  To simplify the final answer we also made a
  substitution $T \leftarrow 2T + 3$.

  In the reverse case we put $a = 3 u s - 3 w^2$ and $b = 3 v s^2 - 3
  u w s + 2 w^3$. We then compute
\begin{align*}
  F_1^-( -3  s^2 T - w^2, s - w, 2 w, 1 ) &= 9 s^3 q_1, \\
  F_2^-( -3 s^2 T - w^2, s - w, 2 w, 1 ) &= 9 s^3 (w q_1 - s q_2),
\end{align*}
where
\begin{align*} \hspace{-0.3em}
  q_1 &= 3 u^2 - u v - 3 u w + 2 (3 T - 1) u s - 6 v w
  - 3 (2 T - 3) v s + 2 w^2 -  3 T^2 w s + 9 T^2 s^2, 
  \hspace{-0.7em} \\ \hspace{-0.3em}
  q_2 &= 3 (T - 2) u^2 + 3 u v + u w - 2 v^2 - 6 (2 T - 3) v w + 3 (6
  T - 1) v s + 9 T^2 w s + 3 T^3 s^2. \hspace{-0.7em}
\end{align*}
The quadric intersection $\{q_1 = q_2 = 0\} \subset \PP^3$ is a smooth
curve of genus one, defined over $\Q(T)$, with rational point
$(u:v:w:s) = (2:1:1:0)$.  The elliptic surface in
Theorem~\ref{thm:ellsurf} is obtained by putting this curve in
Weierstrass form, for example using the method in
\cite[Chapter~8]{Ca}.  To simplify the final answer we also made a
substitution $T \leftarrow (2 T - 3)/(2 T + 1)$.
\end{ProofOf}

\section{Examples}
\label{sec:ex}

We use Theorems~\ref{MainThm9} and~\ref{thm:geom} to find some
non-trivial pairs of $9$-congruent elliptic curves over $\Q$ and
$\Q(T)$.  By ``non-trivial'' we mean that the elliptic curves are not
isogenous. The examples may be checked independently of our work by
comparing traces of Frobenius.

\subsection{Examples over $\Q$}
We refer to elliptic curves over $\Q$ by their labels in Cremona's tables
\cite{Cr}. For elliptic curves beyond the range of 
Cremona's tables we simply write the conductor followed by a ${\tt *}$.

\begin{Example}
\label{ex9-3}
Let $E$ be the elliptic curve $47775z1$. The equations for $X_E(9)$ in
Theorem~\ref{MainThm9} with $a = -41489280$ and $b = 102867483600$ may
be simplified by substituting \small
\[ \begin{pmatrix} x \\ y \\ z \\ t \end{pmatrix} \leftarrow
\begin{pmatrix}
 2520473760 & 937149484320 & -1998984627360 & -152410870080 \\
 0 & 79644600 & -185343480 & -3827880 \\ 
 0 & -22932 & 47040 & 6468 \\ 0 & -6 & 13 & 1 
\end{pmatrix} \begin{pmatrix} x \\ y \\ z \\ t \end{pmatrix}. \]
\normalsize This gives a model for $X_E(9)$ with equations
\begin{align*}
-x^2 z + x^2 t + 4 x y z + 2 x y t - 3 x z^2 + 2 x z t - 3 x t^2  + 6 y^3 + 
14 y^2 z \hspace{6em} 
 & \\ +~y^2 t + 6 y z^2 - 4 y z t + 9 y t^2 - 6 z^3 + 27 z^2 t
 - 13 z t^2 - t^3 & = 0 \\
-3 x^2 y + 4 x^2 z + 3 x^2 t + 3 x y^2 + 20 x y z - 12 x y t - 3 x z^2  - 32 x z t
+ 25 x t^2  + 21 y^3 & \\ + 16 y^2 z - 24 y^2 t - 12 y z^2 + 100 y z t + 34 y t^2 + 
39 z^3 - 21 z^2 t - 56 z t^2 - 11 t^3 & = 0 \rlap{.}
\end{align*}
On this curve we find rational points $P_1 = (1:0:0:0)$, $P_2 =
(4:-1:-1:0)$ and $P_3 = (1:2:-1:0)$. The corresponding elliptic curves
directly $9$-congruent to $E$ are
\begin{align*}
  & P_1 & & 47775z1 & y^2 + y &= x^3 - x^2 - 32013 x + 2215478 \\
  & P_2 & & 429975*  & y^2 + y &= x^3 - 314688780 x - 2148671872069 \\
  & P_3 & & 494901225* & y ^2 + y &= x^3 - 23634650164230 x -
  21037908383222056594
\end{align*}
Since $X_E^-(9)$ is not locally soluble at $p=7$ there are no elliptic
curves reverse $9$-congruent to $E$.
\end{Example}

In addition to Example~\ref{ex9-3} we have found two further triples
of directly $9$-congruent non-isogenous elliptic curves over
$\Q$. These are
\[ \begin{array}{ll}
  4650j1  & y^2 + x y = x^3 + x^2 - 2700 x + 54000  \\
  553350* & y^2 + x y = x^3 + x^2 - 10472207700 x -
  455228489646000  \\
  1966950* & y^2 + x y = x^3 - x^2 - 20654522386242 x -
  36130051534030639084  \bigskip \\
  27606c1  & y^2 + x y = x^3 - 10289707 x + 12703497719  \\
  358878*   & y^2 + x y = x^3 + 2940333 x - 1416695391  \\
  1242270* & y^2 + x y + y = x^3 - x^2 - 359912 x - 322105301
\end{array} \]
The elliptic curves $1701a1$, $1701g1$ and $22113c1$ are also 
$9$-congruent but only the last two of these
are directly $9$-congruent.

\begin{Example}
  Let $E$ be the elliptic curve $201c1$.  The equations for $X^-_E(9)$
  in Theorem~\ref{MainThm9} with $a = -1029699$ and $b = 402173694$
  may be simplified by substituting \small
\[ \begin{pmatrix} x \\ y \\ z \\ t \end{pmatrix} \leftarrow
\begin{pmatrix}
 -26471709 & -23136696 & 20106774 & -20376135 \\ 
  -45147 & -39828 & 33990 & -34509 \\ 
  90294 & 79332 & -68304 & 69342 \\ 77 & 68 & -58 & 59 
\end{pmatrix} \begin{pmatrix} x \\ y \\ z \\ t \end{pmatrix}. \]
\normalsize 
This gives a model for $X_E^-(9)$ with equations
\begin{align*}
-x^3 + 4 x^2 y + 3 x^2 z - x^2 t + 6 x y^2 + 2 x y z - 2 x y t - 6 x z^2 + 
4 x z t \hspace{2em} 
 & \\ - 11 x t^2 + y^3 + 7 y^2 t - 2 y z^2 + 4 y z t - 4 y t^2 + 6 z^3 - 
7 z^2 t + 4 z t^2 + t^3 & = 0 \\
2 x^3 - x^2 y + 5 x^2 t - 10 x y^2 - 2 x y z + 16 x y t - 3 x z^2 + 4 x z t + 
8 x t^2 \hspace{2em} 
 & \\ - 5 y^3 - y^2 z - 3 y^2 t - y z^2 - 2 y z t + 12 y t^2 + 3 z^3 - 4 z^2 t 
+ 2 z t^2 - 3 t^3 & = 0  \rlap{.}
\end{align*}
On this curve we find the rational point
$(1:-2:-1:0)$. The corresponding elliptic curve reverse 
$9$-congruent to $E$ is 
\begin{align*}
& 374865* & y^2 + x y = x^3 + x^2 - 60068738107 x + 4858035498982726.
\end{align*}
\end{Example}

\subsection{Examples over $\Q(T)$}
Again we start our investigations with the equations for
$X_E^{\pm}(9)$ in Theorem~\ref{MainThm9}.  To find some interesting
examples over $\Q(T)$, we set $a=b=-27j/(4(j-1728))$ to obtain a model
for $\ZZ^\pm(9)$, fibred over the $j$-line.  We then looked for some
rational curves on this surface, by intersecting with coordinate
hyperplanes.

The following example gives a proof of Theorem~\ref{thm:inf} in the
direct case, without going via Theorem~\ref{thm:ellsurf}.

\begin{Example} Let $E/\Q(T)$ be the elliptic curve $y^2 = x^3 + a(T)
  x + b(T)$ where
\begin{align*}
  a(T) & = \tfrac{1}{2} (39 T^4 - 60 T^3 - 162 T^2 + 60 T + 39), \\
  b(T) & = 47 T^6 + 120 T^5 + 21 T^4 + 21 T^2 - 120 T + 47.
\end{align*}
Then on $X_E(9)$, with equations as in Theorem~\ref{MainThm9}, we find
the rational point
\[ (x:y:z:t) = (\tfrac{15}{2}(3 T^4 + 8 T^3 - 2 T^2 - 8 T + 3) : T^2
+1 : 1 :0).\] The corresponding elliptic curve directly $9$-congruent
to $E$ is the elliptic curve directly $3$-congruent to $E$ constructed
in Theorem~\ref{MainThm3} with $c_4 = -a(T)/27, c_6 = -b(T)/54$ and
$(r:s) = (47 T^6 - 78 T^5 - 153 T^4 + 244 T^3 + 153 T^2 - 78 T - 47 :
18 (T^2 + 1) (T^2 + 6 T - 1) )$.  Specialising to $T=0$ gives a pair
of elliptic curves with conductors $80640$ and $5886720$. In
particular these curves are not isogenous.
\end{Example}

The following example gives a proof of Theorem~\ref{thm:inf} in the
reverse case, without going via Theorem~\ref{thm:ellsurf}.

\begin{Example}
  Let $E/\Q(T)$ be the elliptic curve $y^2 = x^3 + a(T) x + b(T)$
  where
  \begin{align*}
    a(T) & = 3 (3 T + 1) (6 T^3 - 3 T - 1) (9 T^3 - 9 T - 4)^2, \\
    b(T) & = 2 (3 T^3 + 27 T^2 + 21 T + 4) (6 T^3 - 3 T - 1)^2 (9 T^3
    - 9 T - 4)^2.
  \end{align*}
  Then on $X^-_E(9)$, with equations as in Theorem~\ref{MainThm9}, we
  find the rational point
  \[ (x:y:z:t) = (-(6 T^3-3 T-1) (9 T^3 - 9 T - 4):T:1:0). \] The
  corresponding elliptic curve reverse $9$-congruent to $E$ is the
  elliptic curve reverse $3$-congruent to $E$ constructed in
  Theorem~\ref{MainThm3} with $c_4 = -a(T)/27, c_6 = -b(T)/54$ and $
  (r:s) = ( (3 T+1) (9 T^3-9 T-4) (6 T^3-3 T-1) (180 T^4 + 321 T^3 +
  216 T^2 + 66 T + 8) : 3 (369 T^6 + 1107 T^5 + 1431 T^4 + 1017 T^3 +
  414 T^2 + 90 T + 8) ).$ Specialising to $T=-1/4$ gives the pair of
  elliptic curves 2304o1 and 343296g1.  In particular these curves are
  not isogenous.
\end{Example}

\section*{Acknowledgements}
I would like to thank Zexiang Chen for some useful discussions that in
particular helped to simplify the proof of Theorem~\ref{MainThm9}.
All computer calculations in support of this work were carried out
using Magma~\cite{Magma}.


\begin{thebibliography}{MM}

\frenchspacing
\renewcommand{\baselinestretch}{1}

\bibitem{BS} R. Barman and A. Saikia, A note on Iwasawa
  $\mu$-invariants of elliptic curves, {\em Bull. Braz. Math. Soc.}
  {{41}} (2010), no. 3, 399--407.

\bibitem{Magma} W. Bosma, J. Cannon and C. Playoust, The Magma algebra
  system I: The user language, {\em J. Symb. Comb.} {{24}}, 235--265
  (1997). See also 
  \url{http://magma.maths.usyd.edu.au/magma/}

\bibitem{BHLS} R. Br\"oker, E.W. Howe, K.E. Lauter and P. Stevenhagen,
  Genus-2 curves and Jacobians with a given number of points, to
  appear in {\em LMS J. Comput.
    Math.} 

\bibitem{BD} N. Bruin and K. Doerksen, The arithmetic of genus two
  curves with $(4,4)$-split Jacobians, {\em Canad. J. Math.} {{63}}
  (2011), no. 5, 992--1024.

\bibitem{Ca} J.W.S. Cassels, {\em Lectures on elliptic curves}, London
  Mathematical Society Student Texts, {{24}}, Cambridge University
  Press, Cambridge, 1991.

\bibitem{Chen8} Z. Chen, {\em Families of elliptic curves with the
    same mod $8$ representation}, in preparation.

\bibitem{Cr} J.E. Cremona, {\em Algorithms for modular elliptic
    curves}, Cambridge University Press, Cambridge, 1997.  See also
  \url{http://www.warwick.ac.uk/~masgaj/ftp/data/}

\bibitem{CM} J.E. Cremona and B. Mazur, Visualizing elements in the
  Shafarevich-Tate group, {\em Experimental Mathematics} {9}:1, (2000)
  13-28.

\bibitem{pfaff} T.A. Fisher, Pfaffian presentations of elliptic normal
  curves, {\em Trans. Amer. Math. Soc.} {{362}} (2010), no. 5,
  2525--2540.

\bibitem{g1hess} T.A. Fisher, The Hessian of a genus one curve, {\em
    Proc. Lond. Math. Soc.} (3) {{104}} (2012), no. 3, 613--648.

\bibitem{enqI} T.A. Fisher, Invariant theory for the elliptic normal
  quintic, I. Twists of X(5), {\em Math. Ann.} {{356}} (2013), no. 2,
  589--616.

\bibitem{7and11} T.A. Fisher, On families of 7 and 11-congruent
  elliptic curves, {\em LMS J. Comput. Math.} {{17}} (2014), no. 1,
  536--564.

\bibitem{Frey} G. Frey, On elliptic curves with isomorphic torsion
  structures and corresponding curves of genus $2$, in {\em Elliptic
    curves, modular forms \& Fermat's last theorem} (Hong Kong, 1993),
  J.~Coates and S.-T.~Yau (eds.), Ser. Number Theory~I, Int. Press,
  Cambridge, MA, (1995) 79--98.

\bibitem{HK} E. Halberstadt and A. Kraus, Sur la courbe modulaire
  $X_E(7)$, {\em Experiment. Math.} {{12}} (2003), no. 1, 27--40.

\bibitem{KS} E. Kani and W. Schanz, Modular diagonal quotient
  surfaces, {\em Math. Z.} {{227}} (1998), no. 2, 337--366.

\bibitem{K} M.A. Kenku, On the number of $\Q$-isomorphism classes of
  elliptic curves in each $\Q$-isogeny class, {\em J. Number Theory}
  {{15}} (1982), no. 2, 199--202.

\bibitem{M} B. Mazur, Rational isogenies of prime degree (with an
  appendix by D. Goldfeld), {\em Invent. Math.} {{44}} (1978), no. 2,
  129--162.

\bibitem{Pap6} I. Papadopoulos, Courbes elliptiques ayant m\^eme
  $6$-torsion qu'une courbe elliptique donn\'ee, {\em J. Number
    Theory} {{79}} (1999), no. 1, 103--114.

\bibitem{PSS} B. Poonen, E.F. Schaefer and M. Stoll, Twists of $X(7)$
  and primitive solutions to $x^2+y^3=z^7$, {\em Duke Math. J.}
  {{137}} (2007), no. 1, 103--158.

\bibitem{RubinSilverberg} K. Rubin and A. Silverberg, Families of
  elliptic curves with constant mod $p$ representations, in {\em
    Elliptic curves, modular forms \& Fermat's last theorem} (Hong
  Kong, 1993), J.~Coates and S.-T.~Yau (eds.), Ser. Number Theory~I,
  Int. Press, Cambridge, MA, (1995) 148--161.

\bibitem{RubSil6} K. Rubin and A. Silverberg, Mod $6$ representations
  of elliptic curves, in {\em Automorphic forms, automorphic
    representations, and arithmetic} (Fort Worth, TX, 1996),
  R.S.~Doran, Z.-L.~Dou and G.T.~Gilbert (eds.), Proc. Sympos. Pure
  Math., 66, Part 1, Amer. Math. Soc., Providence, RI, (1999),
  213–-220.

\bibitem{RubinSilverberg2} K. Rubin and A. Silverberg, Mod $2$
  representations of elliptic curves, {\em Proc. Amer. Math. Soc.}
  {{129}} (2001), no. 1, 53--57.

\bibitem{Silverberg} A. Silverberg, Explicit families of elliptic
  curves with prescribed mod $N$ representations, in {\em Modular
    forms and Fermat's last theorem} (Boston, MA, 1995), G.~Cornell,
  J.H.~Silverman and G.~Stevens (eds.), Springer-Verlag, New York,
  (1997), 447--461.

\end{thebibliography}
\end{document}